\documentclass[12pt, reqno]{amsart}
\usepackage{amsmath, amsthm, amssymb, amsfonts, color} 
\usepackage{fullpage}
\usepackage{euscript}
\usepackage{tabularx,multirow}
\usepackage[all]{xy}
\usepackage{upgreek} 
\usepackage[shortlabels]{enumitem}
\usepackage{setspace}
\usepackage{tabularx,multirow}
\usepackage{graphicx}
\usepackage{tikz,pgf}
\usepackage[colorlinks, plainpages]{hyperref}
\usepackage[T1]{fontenc}
\usepackage[english]{babel}
\input{diagrams.tex}
\usetikzlibrary{matrix,arrows,decorations.pathmorphing}
\usetikzlibrary{calc}
\makeatletter
\newbox\dottedarrow@box
\setbox\dottedarrow@box\hbox
  {%
    \begin{tikzpicture}
      \draw[dotted,->] (0,0) -- (1.5em,0);
    \end{tikzpicture}%
  }

\title[Gromov elliptic resolutions of quartic double solids]{Gromov elliptic 
resolutions of quartic double solids}
\author{Ciro Ciliberto and Mikhail Zaidenberg}

\address{Dipartimento di Matematica, Universit\`a di Roma ``Tor Vergata'', Via della 
Ricerca Scientifica, 00177 Roma, Italia}
\email{cilibert@axp.mat.uniroma2.it}
\address{Univ. Grenoble Alpes, CNRS, IF, 38000 Grenoble, France}
\email{mikhail.zaidenberg@univ-grenoble-alpes.fr}

\newtheorem{theorem}{Theorem}[section]
\newtheorem*{theorem*}{Theorem}
\newtheorem*{conjecture*}{Conjecture}

\newtheorem{proposition}[theorem]{Proposition}

\newtheorem*{corollary*}{Corollary}

\theoremstyle{definition}

\newtheorem*{example*}{Example}

\theoremstyle{remark}
\newtheorem*{remark*}{Remark}
\newtheorem{remark}[theorem]{Remark}

\theoremstyle{remark}
\newtheorem*{remarks*}{Remarks}

\newcommand{\ZZ}{{\mathbb Z}}
\newcommand{\PP}{{\mathbb P}}

\newcommand{\CC}{{\mathbb C}}

\newcommand{\QQ}{{\mathbb Q}}
\newcommand{\kk}{{\mathbb K}}

\def\cO{{\mathcal O}}

\newcommand{\A}{{\mathbb A}}

\DeclareMathOperator{\Pic}{Pic}

\def\leq{\leqslant}
\def\ge{\geqslant}
\def\le{\leqslant}
\begin{document}
\begin{abstract}
We construct examples of nodal quartic double solids 
that admit uniformly rational, 
and so elliptic in Gromov' sense, small algebraic resolutions.
\end{abstract}
\thanks{Acknowledgements: The first authors is a members of GNSAGA 
of the Istituto Nazionale di Alta Matematica ``F. Severi''}
\maketitle

%{\footnotesize \tableofcontents}
%
\section{Introduction}
We work over an algebraically closed field $\kk$ of characteristic zero.
A smooth algebraic variety $X$ is called \emph{algebraically 
elliptic in Gromov' sense}, or simply
\emph{Gromov elliptic}, if it admits a \emph{dominating spray} $(E,p,s)$,
where $p\colon E\to X$ is an algebraic vector bundle with a zero section 
$Z$ and 
$s\colon E\to X$ is a morphism such that $s|Z=p|_Z$ and for each fiber 
$E_x=p^{-1}(x)$, $x\in X$,
the restriction $ds|_{T_{0_x}E_x}\colon T_{0_x}E_x\to T_xX$ is surjective,
where $0_x\in E_x$ is the origin of the vector space $E_x$. 
The Gromov ellipticity of a variety 
ensures some important properties of approximation and 
interpolation, see
e.g. \cite{Gro89}, \cite{For17}, \cite{For23} and \cite{Zai24}. 
Furthermore, a Gromov elliptic variety $X$ of dimension $n$ 
admits a morphism $f\colon \A^{n+1}\to X$ 
such that the restriction $f|_U\colon U\to X$ to 
some open subset $U\subset \A^{n+1}$ is smooth and 
surjective, see \cite{Kus22}. For $\kk=\CC$  
there exists a morphism $f\colon \A^{n}\to X$ 
with similar properties provided that $X$ is complete, 
see \cite[Theorem~1.1]{For17a}. 

The following projective varieties are Gromov elliptic, 
see e.g. \cite{AKZ24}, \cite{Zai24} 
and the literature therein:
\begin{itemize}
\item rational  smooth projective surfaces;
\item flag varieties $G/P$ of semisimple algebraic groups;
\item smooth complete spherical varieties, in particular, 
\item smooth complete toric varieties;
\item smooth Fano threefolds of index 2 and genus 10, etc.
\end{itemize}
All varieties in this list are uniformly rational. 
In the terminology of \cite{BB14}, 
an algebraic variety $X$ of dimension $n$ is 
called \emph{uniformly rational} if it is covered by open charts 
isomorphic to open subsets in $\A^n$.  According 
to \cite[Example~2.4]{BB14} every rational
smooth cubic hypersurface in $\PP^{n+1}$, $n\ge 2$, 
is uniformly rational, and so is Gromov elliptic
due to the following result. 
\begin{theorem}[\rm{\cite[Theorem~1.3]{AKZ24}}]
\label{thm:AKZ} 
A uniformly rational complete 
algebraic variety $X$ is Gromov elliptic. 
\end{theorem} 
In fact, every
smooth cubic hypersurface in $\PP^{n+1}$, 
$n\ge 2$, is Gromov elliptic,
see \cite{KZ24}.
In particular, smooth cubic threefolds in $\PP^4$ 
provide example of irrational Gromov elliptic 
projective varieties.

It is plausible that also every smooth 
quartic double solid is 
Gromov elliptic, cf. \cite[Question~4.27]{Zai24}. 
Recall that a nodal quartic double solid 
with at most six 
nodes is irrational,
see \cite[Theorem~1.2]{CPS19}.
In the present paper
we construct examples of rational nodal 
quartic double solids   
with $6+s$ ($4\leq s\leq 10$) nodes, which admit 
 small algebraic resolutions that are  
uniformly rational. Due to Theorem \ref{thm:AKZ}
these resolutions are Gromov elliptic. 
So, our main result is the following theorem. 
\begin{theorem}\label{mthm}
For any $s=4,\ldots,10$,  and $s\neq 5$,  
there are quartic double solids  
with $6+s$ nodes  which admit algebraic 
 small resolutions that are
uniformly rational, and so they are Gromov elliptic. 
\end{theorem}
\section{Cubic threefolds and quartic double solids 
with branch surface
having a trope}\label{sec2}
Let $X \subset \PP^4$ be an irreducible, 
reduced cubic hypersurface with 
$s$ ordinary double points or nodes, 
i.e., singularities of type $A_1$.  
The projection from a node shows that $X$ is 
rational provided $s\ge 1$.
It is well known  that $s \le 10$ 
(see, e.g., \cite{Seg1886}, \cite{Kal86}). 
In what follows we concentrate on the case of 
a rational singular $X$, and so $s\ge 1$.

Let $x\in X$ be a general point and $p\colon Y \to X$ 
be the blow-up of $ X$ at $x$
with exceptional divisor $E \cong \PP^2$. 
The projection of $X$ from $x$ to $\PP^3$ 
determines a surjective morphism $\pi\colon Y \to \PP^3$ 
that is a double cover branched along a  quartic surface 
$B\subset \PP^3$, 
see \cite[Sec. 3]{Kre00}. 
The image of $E$ via $\pi$ is a plane in $\PP^3$ touching 
$B$ along a conic (that we will soon see to be irreducible), 
which passes through exactly six distinct nodes of $B$. 
In the terminology of \cite{Hen11} and \cite{Jes16}, 
such a plane is called a \emph{trope} of $B$.

More in detail, we can choose affine coordinates 
$(x_1,\ldots,x_4)$ in $\PP^4$ such that $x$ 
is the origin and $X$ 
is defined by an affine equation of the form
$x_1 + 2g_2 + g_3 = 0$
where $g_k(x_1, x_2, x_3, x_4) \in \kk[x_1, x_2, x_3, x_4]$ 
is homogeneous of degree $k \in \{2, 3\}$. 
The affine equation of the tangent hyperplane $T_{X,x}$ 
 is $x_1 = 0$. 
The surface $B$ is defined in $\PP^3$ 
with homogeneous coordinates 
$[x_1,x_2,x_3,x_4]$ by the homogeneous equation 
$g^2_2 - x_1g_3 = 0$.
The intersection of $X$ 
with the tangent hyperplane 
$x_1 = 0$ to $X$ at $x$ is the cubic surface $F$ in 
$\PP^3$ with equation
$2g_2(0, x_2, x_3, x_4) + g_3(0, x_2, x_3, x_4) = 0$.
It is smooth off $x$
and has at $x$ an ordinary double point 
(cf. \cite[Thm.~1.4]{CC01}). 
This implies that $g_2(0, x_2, x_3, x_4)$ is 
a quadratic form of rank $3$, 
and so the equation 
$g_2(0, x_2, x_3, x_4)=0$ defines 
a smooth conic $C$ in $\PP^2$. 
There are exactly six distinct lines 
on $F$ passing through $x$. 
In the hyperplane $x_1 = 0$ these lines are defined 
by the system 
 $g_2(0, x_2, x_3, x_4) = g_3(0, x_2, x_3, x_4) = 0$.
 
 The proper transforms on $Y$ of 
 the six lines on $X$ through $x$ 
 are contracted by $\pi$ to the six points in $\PP^3$, 
 with homogeneous coordinates $[x_1, x_2, x_3, x_4]$, 
 defined by the system $x_1 = g_2 = g_3 = 0$. 
 These points are nodes of the surface $B$ 
lying on the conic $C$. 
 The plane of $\PP^3$ 
 with equation $x_1 = 0$ is the trope 
 containing these nodes. 
 The trope  is tangent to $B$ along $C$.
Finally, $B$ has further $s$ distinct nodes that are 
the images of the original $s$ nodes of $X$; 
for all this, see again \cite[Sec. 3]{Kre00}.

One has the following:

\begin{proposition}[\rm{see \cite[Thm.~3.1]{Kre00} or 
\cite[Sec.~3.2]{Had00}}] 
\label{prop:Kre-Had} Giving a double cover
$\phi\colon Z\to \PP^3$ 
 branched along  a quartic surface $B$, 
 with  exactly $6 + s$ nodes 
corresponding to the nodes of $B$, six of 
which lie on a trope, is equivalent to  
giving the following data:
\begin{itemize}
\item a cubic hypersurface $X$ in  $\PP^4$ with $s$ nodes;
\item a smooth point $x\in X$ with six distinct lines 
on $X$ passing through $x$, lying on 
a quadric cone of rank 3, 
with the blow-up $\sigma\colon Y \to X$ of $X$ at $x$;
\item a morphism $\varphi\colon Y \to Z$ 
such that the morphism 
$\tilde\pi\colon Y \to \PP^3$ determined by the projection 
$\pi$ from $x$ factors through $\varphi$ and $\phi$, 
 where
\item $\varphi\colon Y \to Z$ is a partial small resolution 
of $Z$ at the six nodes 
of $Z$ lying on the trope.
\end{itemize}
\end{proposition}
{ \begin{remark}
The varieties and morphisms from Proposition 
\ref{prop:Kre-Had} fit in the following commutative diagram:

\usetikzlibrary{matrix,arrows,decorations.pathmorphing}
     \begin{center}
        \begin{tikzpicture}[scale=2]
        
        \node at (0,1){$Y$};
        \node at (0,0){$Z$};
        \node at (1,1){$X$};
        \node at (2,1){$\PP^4$};
        \node at (1,0){$\PP^3$};
        \node at (0.5,0.2){$\phi$};
        \node at (1.5,1){$\hookrightarrow$};
        \draw[->][thick] (0,0.8)--node[left=1pt]{$\varphi$} (0,0.2); 
        \draw[->][thick] (0.2,1)--node[above=1pt]{$\sigma$} (0.8,1); 
        \draw[->][thick] (0.2,0.8)--node[above=1pt]{$\tilde\pi$} (0.8,0.2);
         \draw[->][thick] (0.2,0)--node[above=1pt]{$\phi$} (0.8,0); 
         \draw[->][thick,dashed] (1,0.8)--node[right=1pt]{$\pi$} (1,0.2); 
         \draw[->][thick,dashed] (1.8,0.8)--node[right=1pt]{$\pi$} (1.2,0.2);

 \end{tikzpicture}
\end{center}
\end{remark}
}
 \section{Proof of Theorem \ref{mthm}}
 We recall first some results from 
\cite{BB14} used in the proof.
\begin{theorem}[\rm{cf. \cite[Prop. 2.6]{BB14}}] \label{thm:blow}
Let $X$ be a uniformly rational variety and let $Z \subset X$ 
be a smooth subvariety. 
Then the blow-up of $X$ along $Z$ is uniformly rational.
\end{theorem}
\begin{theorem}[\rm{cf. \cite[Thm.~3.10]{BB14}}] 
\label{thm:unif-rat}
All small algebraic resolutions of nodal cubic 
threefolds $X$ in $\PP^4$
 are uniformly rational.
\end{theorem}
For the next theorem see \cite[Ch.~III]{Wer87},
\cite[Lem.~1.2 and Sec.~3]{FW89} and  \cite[Lem.~3.4]{BB14}.
\begin{theorem}
 \label{thm:proj-res}
Let $X$ be a nodal cubic threefold in $\PP^4$. 
There exists a projective small resolution of $X$ 
if and only if there exists an 
algebraic small resolution of $X$. 
Furthermore, if there exists a projective small 
resolution of $X$, then 
all small resolutions of $X$ are algebraic\footnote{
However, these are in general only complete 
algebraic varieties which 
need not be projective.}.
\end{theorem}
 One can find in \cite{FW89} and \cite{Wer87} 
{ a complete} classification 
of nodal cubic threefolds $X$ in $\PP^4$ 
that admit projective small resolutions. 
This classification uses a basic invariant 
$d=d(X):={\rm rk} \,{\rm Cl}(X)- {\rm rk}\, {\Pic}(X)$
called the \emph{defect} of $X$.
It can also be defined as the non-negative integer $d$
such that the vector space 
of linear forms that vanish on the $s$ nodes of $X$ 
has dimension $5+ d - s\ge 0$ 
(see \cite[Thm.~4.2]{Wer87}). We have $s\ge d$, 
see  \cite[Sec.~1]{FW89}.
It turns out that if $X$
admits a projective (hence algebraic) small
resolution then it must have defect $d \ge 1$. 
 Moreover, using the classification 
 results in  \cite[Sec.~3]{FW89} 
and \cite[Ch.~VIII, Thm.~8.1 
and Lem.~8.2]{Wer87} one can 
immediately deduce the following: 
\begin{theorem}\label{thm:FW}
If a nodal cubic threefold $X$ in $\PP^4$ admits 
a projective small resolution, then
the pair $(d, s)$ takes one of the values
$(1, 4), (1, 6), (2, 6), (2, 7), (3, 8), (4, 9), (5, 10)$. 
Furthermore, for every pair $(d, s)$ 
from this list there exists a 
nodal cubic threefold $X$ in $\PP^4$ 
with precisely $s$ nodes and with defect $d$ 
that admits a projective small resolution. 
\end{theorem}
 \begin{remark}  
 These results are 
 established in the cited papers 
 in the case where $\kk=\CC$ 
 is the complex number field.
 It is easily seen that they still hold over 
 any algebraically closed field $\kk$ 
 of characteristic zero.
 Indeed, the equation of the corresponding 
 cubic threefold 
$X$ in the list above can be recovered 
 starting from the equation of 
 the \emph{associated curve} $S$
 on $\PP^1\times\PP^1$  of type $(3, 3)$, 
 see \cite{FW89}. 
 We describe this recovering process below. 
 We also show that $X$ is defined over $\bar\QQ$ 
provided $S$ is. Let us show first that one can choose 
$S$ that is defined over $\bar\QQ$. 
 
 Let $X$ be a nodal cubic threefold in $\PP^4$.
Choosing a nodal point $P$  of the threefold $X$, 
in appropriate coordinates $(t_0,\ldots,t_4)$ on $\PP^4$, 
$X$ can be given by equation
$t_4Q+R=0$, where $Q=t_0t_3-t_1t_2$,
$R=R(t_0,\ldots,t_3)$ is a cubic form and $P=(0:0:0:0:1)$.

In $\PP^3$ with coordinates $(t_0,\ldots,t_3)$, $Q$
(resp. $R$)
defines a quadric also noted $Q$
(resp. a cubic surface also noted $R$).
The associated curve $S$ is the complete intersection curve
$Q\cdot R$. This is a $(3,3)$-curve 
on the quadric $Q\cong \PP^1\times\PP^1$.
 
 According to \cite[Lem.~2.4]{FW89}, 
 see also \cite[Thm. 8.1]{Wer87}, 
 a nodal cubic threefold $X$ in $\PP^4$ admits 
a projective small resolution if and only if 
 the associated curve 
 $S$ 
 is nodal, all its irreducible components 
 are smooth, and at least one of these 
 components is of type $(a,b)$ with $a\neq b$. 
 These conditions are satisfied exactly in the 7 
 cases listed in Theorem \ref{thm:FW}.
 
 Letting each component of $S$ run 
 over the corresponding component of 
 the Hilbert scheme of curves on $\PP^1\times\PP^1$,
 the properties above remain valid 
on a dense open subset 
 of the product of these Hilbert schemes. 
 Hence, in each of the 7 cases above
 one can choose a corresponding curve $S$ 
defined over $\bar\QQ$.
For example, if $(d,s)=(1,6)$ then $S=S_1+S_2$
with $S_1$ of type $(1,2)$ and $S_2$ of type $(2,1)$,
see  \cite[Sec. 3]{FW89}. 
One can take e.g. $S_1=\{x_0y_0^2 + x_1y_1^2 = 0\}$ and 
$S_2=\{x_0^2y_0 + x_1^2y_1 = 0\}$ 
in bihomogeneous coordinates $((x_0 : x_1), (y_0 : y_1))$ on 
$\PP^1\times\PP^1$. 

To recover $R$ starting with $S$, notice first 
that $R$  is defined by $S$ only up to
adding to $R$ the cubic form $QL$ 
with an arbitrary linear form $L$ 
 in $t_0,\ldots,t_3$.
Up to isomorphism, 
the substitution $R\mapsto R+QL$ 
actually does not affect $X$.
Indeed, replacing
$t_4Q+R$ by $(t_4+L)Q+R$ and letting $t_4'=t_4+L$
we get the same equation for $X$
in the new coordinates  $(t_0,\ldots,t_3, t_4')$.

Let $H$ be
a hyperplane in $\PP^3$. 
We have $h^0(\cO_{\PP^3}(3H))=20$. However, 
since $R$ is defined by $S$ only up to adding 
a product  $QL$, up to scaling the dimension 
of the parameter space 
of complete intersections $S=Q\cdot R$ equals 
$20-h^0(\cO_{\PP^3}(H))=16$.

Let $e_0$ and $e_1$ be transversal rulings on $Q$
whose classes generate $\Pic(Q)=\ZZ^2$. 
By the Riemann-Roch and the Kodaira 
vanishing theorems  
we have $h^0(\cO_Q(3(e_0+e_1)))=16$.

Since the above dimensions coincide,
we conclude that every effective 
$(3,3)$-divisor $S$  on $Q$ 
is a complete intersection of $Q$ with
a cubic surface $R$ in $\PP^3$. 
For any member $R'$
of the linear system 
generated by $R$ and $|Q+H|$ 
we also have $S=Q\cdot R'$.

Given a $(3,3)$-curve $S$ on $Q$ 
satisfying the conditions of  \cite[Lem.~2.4]{FW89}
and any cubic surface 
$R$ such that $S=R\cdot Q$, the equation $t_4Q+R=0$
defines a nodal cubic threefold in $\PP^4$ 
of one of the 7 types listed in Theorem \ref{thm:FW}.
The nodes of $X$ are in one to one correspondence 
with the nodes of $S$,
see e.g. \cite[Thm.~1.4(i)]{Vik23}. 

Suppose now that the associated curve $S$ on $Q$
is defined over $\bar \QQ$.
We claim that one can also choose $R$ 
defined over $\bar \QQ$.
Indeed, the embedding 
$Q=\PP^1\times\PP^1 \hookrightarrow \PP^3$
can be given in coordinates via
\[((x_0 : x_1), (y_0 : y_1))\mapsto (t_0,\ldots,t_3)=
(x_0y_0 : x_0y_1 : x_1y_0 : x_1y_1).\]
A monomial
$t_0^a t_1^b t_2^c t_3^d$ of degree $a+b+c+d=3$
yields a bivariate monomial
\[x_0^{a+b} x_1^{c+d} y_0^{a+c} y_1^{b+d} =
 x_0^A x_1^B y_0^C y_1^D \]
of bidegree $(A+B, C+D)=(3,3)$.
It is impossible to recover 
uniquely the exponents 
$(a,b,c,d)$ from the ones $(A,B,C,D)$;
in fact, the matrix of the corresponding system 
of linear equations is degenerate.
This can be done by hand, however.
For instance,  for every
bivariate monomial $x_0^A x_1^B y_0^C y_1^D$ 
with $A > B, C > D$  and $A \ge C$
we get a solution by letting $c=0, a=C, d=B$ and 
$b=D-B=A-C$; 
indeed, we have $a+b+c+d=3$.
We proceed likewise in all other cases. 
In the majority of cases we get a unique solution; 
in the remaining cases there are exactly two solutions, and their 
difference is divisible by the quadric form $Q=t_0t_3-t_1t_2$.
So, starting with a bihomogeneous equation of $S$
with coefficients in $\bar \QQ$, in this way
we get an equation of $R$  
with coefficients in $\bar \QQ$ such that $S=Q\cdot R$.
The resulting cubic threefold $X\subset\PP^4$ 
with equation $x_4Q+R=0$ also is defined over $\bar \QQ$.
\end{remark}

 We are now ready to prove our main result.
\begin{proof}[Proof of Theorem \ref{mthm}]
Let $X$ be a cubic threefold in $\PP^4$ with $s$ 
nodes and defect $d$,  
admitting an algebraic small resolution, see
Theorem \ref{thm:FW}.
Then all small  resolutions of $X$ are algebraic 
(see Theorem \ref {thm:proj-res}) and they are 
uniformly rational by Theorem  \ref {thm:unif-rat}.  
Let $x\in X$ be a general point and let $Y$ be the
blow--up of $X$ at $x$. By Theorem \ref {thm:blow}, 
all small resolutions of $Y$ are also algebraic and 
uniformly rational. 

As in  Proposition \ref {prop:Kre-Had}, 
there is a double cover $\phi: Z\longrightarrow \PP^3$
branched along a quartic surface $B$ with $6+s$ nodes, 
six of which lying on a trope, with
a morphism $\varphi\colon Y \to Z$ such that the morphism 
$\tilde\pi\colon Y \to \PP^3$ determined 
by the projection 
$\pi$ from $x$ factors through $\varphi$ and $\phi$. 
Moreover 
$\varphi\colon Y \to Z$ is a partial small resolution 
of $Z$ at the six nodes 
of $Z$ lying on the trope. 
Then all small resolutions of $Y$ determine 
small resolutions of $Z$ that are algebraic 
and uniformly rational, and therefore Gromov elliptic. 
Now the assertion follows.
\end{proof}
\noindent {\bf Acknowledgments.} 
The second author thanks Yuri Prokhorov and Ivan Cheltsov 
for clarifying discussions. 

\end{document}